\documentclass{article}
\usepackage{amsmath}
\usepackage{amssymb}
\usepackage{amsfonts}
\usepackage{amsthm}

\newtheorem{cor}{Corollary}

\newtheorem{prp}{Proposition}
\newtheorem{rem}{Remark}
\title{On Powers of Some Power Series}
\author{{Milan Janji\'c}}%
\date{}
\begin{document}
\maketitle
\begin{center}
Department of Mathematics and Informatics\\
 University of Banja Luka\\
78000 Banja Luka, Republic of Srpska, BA\\
email: agnus@blic.net
\end{center}

\begin{abstract}
We investigate some relationships between powers of powers series
and compositions of positive integers. We show that the compositions  may be interpreted in terms of  powers of some power series, over arbitrary commutative ring.  As  consequences,
several closed formulas for the compositions as well as for the
generalized compositions with a fixed number of parts are derived.
Some results on compositions obtained in some recent papers are
consequences of these formulas.
\end{abstract}
\section{Introduction} Let $R$ be a commutative ring with $1,$ and
let $\mathbf r=(r_0,r_1,\ldots)$ be any sequence of elements of
$R.$ Let $k,n$ be nonnegative integers such that $k\leq n.$  We
define the finite sequence $C^{(\mathbf r)}(n,k),\;(k=0,1,\ldots,n)$ of elements of $R$ in the following
way:
\[C^{(\mathbf r)}(n,1)=r_{n-1},C^{(\mathbf r)}(0,0)=1,C^{(\mathbf r)}(n,0)=C^{(\mathbf r)}(n,k)=0,(0<n<k),\]
and,
\begin{equation}\label{kgk}C^{(\mathbf r)}(n,k)=\sum_{i=0}^{n-k}r_{i}C^{(\mathbf r)}(n-i-1,k-1),(1<k\leq n).\end{equation}

We also define the sequence $C^{(\mathbf r)}(n),\;(n=0,1,2,\ldots)$ to be
 \begin{equation}\label{svegk}C^{(\mathbf r)}(0)=1,\;C^{(\mathbf r)}(n)=\sum_{k=1}^nC^{(\mathbf r)}(n,k),\;(n\geq 1).\end{equation}

For a fixed positive integer $k$ we let ${\mathbf r,k\choose
n},\;(n=0,1,\ldots)$  denote the sequence which a generating
function is $\left(\sum_{i=0}^\infty r_ix^i\right)^k.$ Thus,
\begin{equation}\label{gf}\left(\sum_{i=0}^\infty
r_ix^i\right)^k=\sum_{n=0}^\infty {\mathbf r,k\choose
n}x^n.\end{equation}
When the $b$'s are nonnegative integers, then $C^{(\mathbf r)}(n,k)$ is the number of generalized compositions of $n$ into $k$ parts
\cite{mil2}, Theorem 1). Also,  It is proved in \cite{mil1} that, in this case, $C^{(\mathbf r)}(n)$ is the number of all generalized compositions of $n.$

In Section 2 we derive a recurrence equation for $C^{(\mathbf r)}(n).$ Further we find a generating function for $C^{(\mathbf r)}(n,k).$ In Section 3 we apply the results obtained in Section 2, on usual compositions. Finally, in Section 4 we derive some results for generalized compositions.
\section{A generating function}
\begin{prp} Let $n$ be a positive integer, and let $\mathbf r=(r_0,r_1,\ldots)$ be a sequence in $R.$  Then
 \begin{equation}\label{allr}C^{(\mathbf r)}(n)=\sum_{i=0}^{n-1}r_{i}C^{(\mathbf r)}(n-i-1).\end{equation}
\end{prp}
\begin{proof} We use induction on $n.$ For $n=1$ the assertion is clearly true. Suppose that it is true for $k<n.$
We have
\[C^{(\mathbf r)}(n)=\sum_{k=1}^nC^{(\mathbf r)}(n,k)=\sum_{k=1}^n\sum_{i=0}^{n-k}r_{i}C^{(\mathbf r)}(n-i-1,k-1).\]
Changing the order of summation we obtain
\[C^{(\mathbf r)}(n)=r_{n-1}+\sum_{i=0}^{n-2}\sum_{k=1}^{n-i}r_{i}C^{(\mathbf r)}(n-i-1,k-1).\]
Taking into account that $C^{(\mathbf r)}(m,0)=0$ if $m>0,$ and then applying the induction hypothesis we get
\[C^{(\mathbf r)}(n)=r_{n-1}+\sum_{i=0}^{n-2}r_iC^{(\mathbf r)}(n-i-1),\] that is
\[C^{(\mathbf r)}(n)=\sum_{i=0}^{n-1}r_iC^{(\mathbf r)}(n-i-1),\] and the proposition is true.
\end{proof}

Next, for a positive integer $k,$ by  ${\mathbf r,k\choose
n},\;(n=0,1,\ldots)$ we denote the sequence which a generating
function is
\begin{equation}\label{gf0}g(x,k)=\left(\sum_{i=0}^\infty
r_{i}x^i\right)^k.\end{equation}

Hence,
\begin{equation}\label{oz}\left(\sum_{i=0}^\infty r_{i}x^i\right)^k=\sum_{n=0}^\infty{\mathbf r,k\choose n}x^n.\end{equation}
Thus,
\begin{equation}\label{kv}{\mathbf r,k\choose n}=\sum_{i_1+i_2+\cdots+i_k=n}r_{i_1}r_{i_2}\cdots r_{i_k},\end{equation}
where the sum is taken over $i_t\geq 0,\;(i=1,2,\ldots,k).$

\begin{prp}\label{pp2} Let $n,k$ be nonnegative integers such that $k\leq n.$ Then,
\begin{equation}\label{t1} C^{(\mathbf r)}(n,k)={\mathbf r,k\choose n-k}.\end{equation}
\end{prp}
\begin{proof}
For $n=k$ we have $C^{(\mathbf r)}(k,k)={\mathbf r,0\choose
k}=r_0^k,$ and the assertion is true. Assume that $k<n.$ We shall
prove  the proposition  by induction on $k.$
\begin{equation}\label{t2} C^{(\mathbf r)}(n,k)={\mathbf r,k\choose n-k}.\end{equation}

For $k=1$ equation (\ref{oz}) takes the form:
\[r_0+r_1x+\cdots={\mathbf r,1\choose 0}+{\mathbf r,1\choose 1}x+\cdots.\]
Since $C^{(\mathbf r)}(n,1)=r_{n-1},$ we have
\[C^{(\mathbf r)}(n,1)={\mathbf r,1\choose n-1},\;(1\leq n),\] and the assertion is true for $k=1.$

Assume that $k>1,$ and that  the
assertion is true for $k-1.$
We obviously have
\[\left(\sum_{i=0}^{\infty}r_{i}x^i\right)^k=\left(\sum_{i=0}^{\infty}r_{i}x^i\right)\cdot\left( \sum_{i=0}^{\infty}r_{i}x^i\right)^{k-1}.\]
Using (\ref{oz}) we may write this  equation in the form:
\[\sum_{i=0}^{\infty}{k,\mathbf r\choose i}x^i=\sum_{i=0}^{\infty}r_ix^i\cdot\sum_{i=0}^{\infty}{\mathbf r,k-1\choose i}x^i,\] so that,
\[\sum_{i=0}^{\infty}{\mathbf r,k\choose i}x^i=\sum_{i=0}^{\infty}\sum_{j=0}^ir_j{\mathbf r,k-1\choose i-j}x^i.\]

By comparing  the terms of $x^{n-k},$ in the preceding equation,  we obtain
\[{k,\mathbf r\choose n-k}=\sum_{j=0}^{n-k}r_j{\mathbf r,k-1\choose n-k-j}.\]
Using the induction hypothesis we obtain
\[{k,\mathbf r\choose n-k}=\sum_{j=0}^{n-k}r_{j}C^{(\mathbf r)}(n-j-1,k-1),\]
and the assertion follows from (\ref{kgk}).
\end{proof}
As an immediate consequence we have
 \begin{equation}\label{fb}C^{(\mathbf r)}(n)=\sum_{k=1}^n{\mathbf r,k\choose n-k}.\end{equation}

\section{Some results on  compositions}
In this section we assume that each $b_i$ is either $1$ or $0.$
It is proved in \cite{mil3} that  then $C^{(\mathbf r)}(n,k)$ counts the number of compositions of $n$ into $k$ parts, all of which are in the set $\{i:b_i=1\}.$ 
Also, in this case the generating function $g(x,k)$ is similar to the function defined in \cite{h2}. 
The number $C^{(\mathbf r)}(n)$ counts the number of all such compositions.
\begin{prp}\label{pp1}
 \begin{enumerate}
 \item[1.] If $r_i=1,\;(i=0,1,\ldots)$, then ${\mathbf r,k\choose n}$ count the number of nonnegative solutions of the Diophant equation \[i_1+i_2+\cdots+i_k=n.\] Hence,
\[{\mathbf r,k\choose n}={n+k-1\choose n}.\]

 \item[2.] Assume that there is a positive number $m$ such that  $r_0=r_1=\ldots=r_{r-1}=1,\;r_i=0,\;(i>m).$
 Then ${\mathbf r,k\choose n}$ counts the number of nonnegative  solutions of the Diophant equation
\[i_1+i_2+\cdots+i_k=n,\] in which each term is $\leq m.$
\end{enumerate}
\end{prp}
\begin{proof}
1. The first assertion follows from (\ref{kv}). The second assertion follows from the well-known fact that the  Diophant equation in 1. has exactly  ${n+k-1\choose n}$ solutions.

2. The assertion follows from (\ref{kv}).
\end{proof}
In the next result we prove that, when each part of a composition belong to a fixed finite set, containing $1,$ then the number of compositions is a sum of multinomial coefficients. The case when $1$ does not belong to the set is easily reduced to this case.     
\begin{prp} Assume that $r_i=1$ if $i\in \{1,m_1,\ldots,m_s\},$ and $r_i=0$ otherwise. Then
\[C^{(\mathbf r)}(n,k)=\sum_{(i_0,i_1,\ldots,i_s)}{k\choose i_0,i_1,\ldots,i_s},\]
where the  sum on the right-hand side is taken over all nonnegative solutions of the Diophant equations
\[i_0+i_1+\cdots+i_s=k,\;(m_1-1)i_1+(m_2-1)i_2+\cdots+(m_s-1)i_s=n-k.\]
\end{prp}
\begin{proof}
In this case,  the formula (\ref{fb}) becomes
\[g(x,k)=\left(1+x^{m_1-1}+x^{m_2-1}+\cdots+x^{m_s-1}\right)^k.\]
Using the multinomial formula yields
\[g(x,k)=\sum_{(i_0,i_1,\ldots,i_s)}{k\choose i_0,i_1,\ldots,i_s}x^{(m_1-1)i_1+\cdots(m_s-1)i_s},\]
where the sum oh the right-hand side is taken over all nonnegative solutions of the Diophant equation
\[i_0+i_1+\cdots+i_s=k.\]
and the proposition follows from Proposition \ref{pp2}.
\end{proof}
The simplest case is $s=1,\;r_1=2.$
\begin{cor} The number of the compositions of $n$ with $k$ parts, each equals either $1$ or $2$ is ${k\choose n-k}.$
\end{cor}
Next, suppose that $s=1,\;r_1=m>2.$ Then,
\[C^{(\mathbf r)}(n,k)=\sum_{(i_0,i_1)}{k\choose i_0,i_1},\]
where $i_0+i_1=k,\;(m-1)i_1=n-k.$ The last equation has a solution if and only if $m-1$ divides $n-k.$

We thus obtain the closed formulas for the number of compositions into $k$ parts and for the number of all compositions.
\begin{cor} Let $n,k,m$ be positive integers with $m>1$. For the number $C^{(\mathbf r)}(n,k)$ of compositions of $n$ with $k$ parts, all of which are either $1$ or $m$ we have is
\[C^{(\mathbf r)}(n,k)={k\choose \frac{n-k}{m-1}},\] if $m-1$ divides $n-k,$ and $0$ otherwise.
For the number of all such compositions we have
\[C^{(\mathbf r)}(n)=\sum_{j=0}^{\lfloor\frac{n}{m}\rfloor}{n-(m-1)j\choose j}. \]
\end{cor}
\begin{rem} This kind of compositions is investigated in \cite{c1}.
\end{rem}

\begin{prp} Let $n,m$ be positive integers, and let $r_i=1$ if $i=1(\mod m),$ and $r_i=0$ otherwise.
Then
\[C^{(\mathbf r)}(n,k)={\frac{n-k}{m}+k-1\choose\frac{n-k}{m}}, \mbox{ if }n=k(\mod m),\] and $C^{(\mathbf r)}(n,k) =0,$ otherwise.
\end{prp}
\begin{proof}
According to (\ref{pp1}), for the usual composition the formula (\ref{oz}) has the form
\begin{equation}\label{gf1}\left(\sum_{i=0}^\infty x^i\right)^k=\sum_{n=0}^\infty {n+k-1\choose n}x^n.\end{equation}
The proposition follows by replacing $x$ by $x^m.$
\end{proof}
\begin{cor} Let $n$ be a positive integer. Then,
\[F_n=\sum_{k}{\frac{n-k}{2}+k-1\choose\frac{n-k}{2}},\] where the sum is taken over $k,$  such that $k$ and $n$ are of the same parity.
\end{cor}
\begin{proof}
For $m=2,$  the numbers $C^{(\mathbf r)}(n,k)$  from the preceding
proposition count the numbers of compositions with odd parts.
Hence, $C^{(\mathbf r)}(n,k)={\frac{n-k}{2}+k-1\choose k-1},$ if
$n$ and $k$ are of the same parity, and $C^{(\mathbf r)}(n,k)=0,$
otherwise. On the other hand, it is proved in Corollary 26, in \cite{mil1},
that the number of all compositions of $n$ is $F_n.$
\end{proof}

In the paper \cite{c3} the compositions without part equals $2$ are considered. We have derived a closed formula for the number of such compositions with $k$ parts. From the formula all results from Theorem 4 to Theorem 9, in \cite{c3}, may be derived.
\begin{prp} Let $n,k$ be positive integers such that $k\leq n.$ If $c_k(n)$ is the number of compositions of $n$ with $k$ parts,  all different of $2,$ then
\[c_n(n)=1,\;c_{n-1}(n)=0,\;c_k(n)=\sum_{i=1}^{\lfloor \frac{n-k}{2}\rfloor}{k\choose i}{n-k-i-1\choose i-1},\;(k<n-1).\]
\end{prp}
\begin{proof}
Equation (\ref{oz}), in this case, take the form:
\[\left(1+x^2+x^3+\cdots\right)^k=\sum_{n=0}^\infty {\mathbf r,k\choose n}x^n.\]
Using the binomial theorem and (\ref{gf1}) we obtain
\[1+\sum_{i=1}^k\sum_{j=0}^\infty{k\choose i}{j+i-1\choose j}x^{2i+j}=\sum_{n=0}^\infty {\mathbf r,k\choose n}x^n.\]
It follows that
\[{\mathbf r,k\choose 0}=1,{\mathbf r,k\choose 1}=0,\;{\mathbf r,k\choose n}=\sum_{2i+j=n}{k\choose i}{j+i-1\choose j},(k>1).\]
The propositions now follows from Proposition \ref{pp2}.
\end{proof}
\begin{cor}[Theorem 9, \cite{c3}] The number $C_k(k+8)$ of $k+8$ into $k$ parts, all of which are different from $2$, is
\[c_k(k+8)=k+5{k\choose 2}+6{k\choose 3}+{k\choose 4}.\]
\end{cor}
\begin{proof} In this case we have 
\[c_k(k+8)=\sum_{i=1}^{4}{k\choose i}{8-i-1\choose i-1},\] and the corollary is true.
\end{proof}

Similarly, we obtain the formula for the compositions without parts equal $m,$ for any $m.$
\begin{prp} Let $n,k$ be positive integer such that $k\leq n,$ let $m$ be a nonnegative integer, and $b_{m}=0,\;b_i=1,\;(i\not=m).$
 \begin{equation}\label{hub1}C^{(\mathbf r)}(n,k)=\sum_{i=0}^k\sum_{i_0+i_1+\cdots+i_{m-1}=k-i}{k\choose i}{k-i\choose i_0,i_1,\ldots,i_{m-1}}{N\choose i-1},\end{equation}
where $N=n-k-ri-1-i_1-\cdots-(m-1)i_{m-1}.$
\end{prp}
\begin{proof}
We consider here the compositions without parts equal $m+1.$  The generating function for the number of such compositions with $k$ parts is
\[g(x,k)=\left[1+x+\cdots+x^{m-1}+x^{m+1}(1+x+x^2+\cdots)\right]^k,\] that is,
\[g(x,k)=\sum_{i=0}^k{k\choose i}\big(1+x+\cdots+x^{m-1}\big)^{k-i}\big(1+x+\cdots\big)^ix^{(m+1)i}.\]
Using (\ref{gf1}) and the multinomial formula yields
\[g(x,k)=\sum_{i=0}^k\sum_{j=0}^\infty\sum_{i_0+i_1+\cdots+i_{m-1}=k-i}X,\]
where
\[X={k\choose i}{k-i\choose i_0,i_1,\ldots,i_{m-1}}{j+i-1\choose j}x^{(m+1)i+j+i_1+2i_2+\cdots+(m-1)i_{m-1}}.\]
It follows that
\[{\mathbf r,k\choose n}=\sum_{i=0}^k\sum_{i_0+i_1+\cdots+i_{m-1}=k-i}{k\choose i}{k-i\choose i_0,i_1,\ldots,i_{m-1}}{N\choose i-1},\]
where $N=n-mi-1-i_1-\cdots-(m-1)i_{m-1},$
and the proposition follows.
\end{proof}
From the preceding proposition we may obtain several closed formula for the numbers of compositions considered in Section 5, \cite{c2}.
On of them is the following:
\begin{cor} Let $n,k$ be positive integers such that $k\leq n.$ Then the number $c^{(3)}_k(n)$ of the compositions of $n$ into $k$ parts with
no occurrence of $3$ is
\[c^{(3)}_k(n)=\sum_{j=0}^k\sum_{i=0}^{\min\{k-j,\lfloor\frac{n-2k+j}{2}\rfloor\}}{k\choose i}{k-i\choose j}{n-2k+j-i-1\choose i-1}.\]
\end{cor}
\begin{proof} This is the case when $m=2.$ Equation (\ref{hub1}) becomes
\[c^{(3)}_k(n)=\sum_{i=0}^k\sum_{i_0+i_1=k-i}{k\choose i}{k-i\choose i_0,i_1}{n-k-2i-1-i_1\choose i-1},\] that is,
\[c^{(3)}_k(n)=\sum_{i=0}^k\sum_{j=0}^{k-i}{k\choose i}{k-i\choose j}{n-2k+j-i-1\choose i-1},\]
 if the conditions $n-2k+j-i-1\geq i-1$ is fulfilled.  Changing the order of the summation and using the preceding condition we conclude that the corollary is true.
\end{proof}
\section{Some results for generalized compositions}
It is proved in \cite{mil3}  that in the case when each $r_i$ is a
nonnegative integer, then $C^{(\mathbf r)}(n,k)$ is the number of
the generalized compositions of $n$ with $k$ parts. Also,
$C^{(\mathbf r)}(n)$ is the number of all generalized compositions
of $n.$ In the paper \cite{g1}, the author considered the compositions when
there are two types of one, and one type of all other natural
numbers. But, no a formula for the number of such
compositions with a fixed number of parts is derived. We start
this section with such a formula, but in a more general case when there are any number of ones.
\begin{prp} Let $n,k$ be nonnegative integers such that $k\leq n,$  let $m$ be a nonnegative integer, and let  $r_0=1+m,\;r_i=1,\;(i>0).$ Then,
\[C^{(\mathbf r)}(n,n)=(1+m)^k,\;C^{(\mathbf r)}(n,k)=\sum_{i=1}^k{k\choose i}{n-k+i-1\choose i-1}m^{k-i},\;(k<n).\]
\end{prp}
\begin{proof} In this case we have
\[g(k,x)=\left(m+1+x+x^2+\cdots\right)^k.\] Using binomial formula we obtain
\[g(k,x)=\sum_{i=0}^k{k\choose i}m^{k-i}\left(1+x+x^2+\cdots\right)^i.\]
Equation xx implies
\[g(k,x)=m^k+\sum_{i=1}^k\sum_{j=0}^\infty{k\choose i}{j+i-1\choose j}m^{k-i}x^j.\]
It follows that
\[{\mathbf r,k\choose 0}=(1+m)^k,\;{\mathbf r,k\choose n}=\sum_{i=1}^k{k\choose i}{n+i-1\choose i-1}m^{k-i},\;(n\geq 1),\]
and the proposition is true.

In particular, if $m=1$, the in \cite{g1}  is proved that $F_{2n+1}$ is the number of all compositions of this kind. Thus,
\begin{cor} The following identity is true
\[F_{2n+1}=2^n+\sum_{i=1}^{n-1}{k\choose i}{n-k-i-1\choose i-1},\] where $F_{2n+1}$ is the Fibonacci number.
\end{cor}

\end{proof}
\begin{prp} Let $n,k,m$ be positive integer such that $k\leq n,$ and let $r_i={m\choose i},\;(i=0,1,\ldots).$ Then,
 \[C^{(\mathbf r)}(n,k)={mk\choose n-k}.\]
\end{prp}
\begin{proof} We have
\[g(k,x)=\left(\sum_{i=0}^m{m\choose i}x^i\right)^k.\] Using the binomial formula implies
\[g(k,x)=(1+x)^{mk},\] and the proposition is true.
\end{proof}
In the next result we show that when the $b$'s are pyramidal numbers then the number of compositions is a pyramidal number also.
\begin{prp} Let $n,k,m$ be positive integer such that $k\leq n,$ and let $r_i={i+m-1\choose i},\;(i=0,1,\ldots).$
Then,
 \[C^{(\mathbf r)}(n,k)={n-k+mk-1\choose n-k}.\]
\end{prp}
\begin{proof} From (\ref{gf1}) follows
\[\left[\sum_{n=0}^\infty {n+k-1\choose n}x^n\right]^m=\left(\sum_{i=0}^\infty x^i\right)^{km},\]
and the proposition is true according to (\ref{oz}).
\end{proof}
In the case $m=2$ we have $r_i=i+1,\;(i=0,1,\ldots),$ so that,
 \[C^{(\mathbf r)}(n,k)={n+k-1\choose 2k-1},\] which is the result of Theorem 1, \cite{ag}.

\begin{prp} Let $n,k$ be positive integers such that $k\leq n,$ and let $r_i=(i+1)^2,\;(i=0,1,\ldots).$ Then
\[C^{(\mathbf r)}(n,k)=\sum_{i=0}^{n-k}{k\choose i}{n-k+2i-1\choose 3i-1}.\]
\end{prp}
\begin{proof}It is a well-known that the generating function $f(x)$  of the sequence $1,2^2,3^2,\ldots$  is
\[f(x)=(1+x)(1+x+x^2+\cdots)^3.\]
It follows that
\[\big[f(x)\big]^k=\sum_{i=0}^\infty\sum_{j=0}^i{k\choose j}{i-j+3k-1\choose i-j}x^{i}.\]
According to (\ref{oz}) we conclude that
 \[C^{(\mathbf r)}(n,k)=\sum_{i=0}^n{k\choose i}{n+2i-1\choose 3i-1},\]
 and the proposition is proved.
\end{proof}

\end{document}